\documentclass[letterpaper,12pt,reqno]{amsart}
\usepackage{graphicx,amsmath,amssymb}
\usepackage{srcltx}
\usepackage[latin1]{inputenc}
\usepackage[T1]{fontenc}
\usepackage{fullpage}
\newtheorem{X}{X}[section]

\newtheorem{lemma}[X]{Lemma}

\newtheorem{theorem}[X]{Theorem}
\newtheorem{conjecture}[X]{Conjecture}

\theoremstyle{definition}
\newtheorem{remark}[X]{Remark}
\newcommand{\V}{\text{Var}}
\newcommand{\E}{\mathbb E}

\renewcommand{\P}{\text{Prob}}
\newcommand{\F}{\mathcal F}
\newcommand{\R}{\mathbb R}

\newcommand{\Z}{\mathbb Z}
\newcommand{\sumchi}{\sum_{\substack{ \chi\bmod q \\ \chi^2=\chi_0 \\ \chi\neq \chi_0}}}
\renewcommand{\a}{\overrightarrow{a}}
\newcommand{\al}{\overrightarrow{\alpha}}
\newcommand{\be}{\overrightarrow{\beta}}

\providecommand{\norm}[1]{\lVert#1\rVert}

\title{Highly biased prime number races}
\author{Daniel Fiorilli}
\address{School of Mathematics, Institute for Advanced Study, 1 Einstein Drive, Princeton NJ 08540 USA}
\curraddr{Department of Mathematics, University of Michigan, 530 Church Street, Ann Arbor MI 48109 USA}
\email{fiorilli@umich.edu}

\begin{document}

\begin{abstract}
Chebyshev observed in a letter to Fuss that there tends to be more primes of the form $4n+3$ than of the form $4n+1$. The general phenomenon, which is referred to as Chebyshev's bias, is that primes tend to be biased in their distribution among the different residue classes $\bmod q$. It is known that this phenomenon has a strong relation with the low-lying zeros of the associated $L$-functions, that is if these $L$-functions have zeros close to the real line, then it will result in a lower bias. According to this principle one might believe that the most biased prime number race we will ever find is the Li$(x)$ versus $\pi(x)$ race, since the Riemann zeta function is the $L$-function of rank one having the highest first zero. This race has density $0.99999973...$, and we study the question of whether this is the highest possible density. We will show that it is not the case, in fact there exists prime number races whose density can be arbitrarily close to $1$. An example of race whose density exceeds the above number is the race between quadratic residues and non-residues modulo $4849845$, for which the density is $0.999999928...$ We also give fairly general criteria to decide whether a prime number race is highly biased or not. Our main result depends on the General Riemann Hypothesis and on a hypothesis on the multiplicity of the zeros of a certain Dedekind zeta function. We also derive more precise results under a linear independence hypothesis.
\end{abstract}
\maketitle
\section{Introduction and statement of results}

The study of prime number races started in 1853, when Chebyshev noted in a letter to Fuss that there seemed to be more primes of the form $4n+3$ than of the form $4n+1$. More precisely, Chebyshev claims without proof that as $c\rightarrow 0$, we have 
$$ -\sum_{p} \left( \frac{-4}p \right)e^{-pc} = e^{-3c}- e^{-5c}+e^{-7c}+e^{-11c}-e^{-13c}-\dots \longrightarrow \infty.$$ 
However, as Hardy and Littlewood \cite{HaLi} and Landau \cite{Lan1,Lan2} have shown, this statement is equivalent to the Riemann hypothesis for $L(s,\chi_{-4})$, where $\chi_{-4}$ denotes the primitive character modulo $4$. 

The modern way to study this question is to look at the set of integers $n$ for which $\pi(n;4,3)>\pi(n;4,1)$, which we denote by $P_{4;3,1}$. One would like to understand the size of this set, however it is known that its natural density does not exist \cite{Ka}. To remedy to this problem we define the \emph{logarithmic density} of a set $P\subset \mathbb N$ by
$$ \delta(P):=\lim_{N\rightarrow \infty} \frac 1{\log N} \sum_{\substack{n\leq N \\ n \in P}} \frac 1n,$$
if the limit exists. In general we define $\underline{\delta}(P)$ and $\overline{\delta}(P)$ to be the $\liminf$ and $\limsup$ of this sequence.
If $P=P_{4;3,1}$, then this last limit exists under the assumption of the Generalized Riemann Hypothesis (GRH) and the Linear Independence Hypothesis (LI), and equals $0.9959...$ (see \cite{RubSar}). 

The \textbf{General Riemann Hypothesis} states that for every primitive character $\chi \bmod q$, all non-trivial zeros of $L(s,\chi)$ lie on the line $\Re (s)=\frac 12$. 

The \textbf{Linear Independence Hypothesis} states that for every fixed modulus $q$, the set 
$$ \bigcup_{\substack{\chi \bmod q\\ \chi \text{ primitive}}} \{ \Im(\rho_{\chi}) : L(\rho_{\chi},\chi)=0, 0<\Re(\rho_{\chi})<1, \Im(\rho_{\chi})\geq 0 \}$$
is linearly independent over $\mathbb Q$.

For a good account of the history of the subject as well as recent developments, the reader is encouraged to consult the great expository paper \cite{MaGr}.  

Rubinstein and Sarnak developed a framework to study this question and more general "prime number races". Assuming GRH and LI, they have shown that for any $r$-tuple $(a_1,\dots a_r)$ of admissible residue classes $\bmod q$ (that is $(a_i,q)=1$), the logarithmic density of the set $P_{q;a_1,\dots,a_r}:= \{n : \pi(n;q,a_1)>\pi(x;q,a_2)>\dots>\pi(x;q,a_r) \}$, which we denote by $\delta(q;a_1,\dots,a_r)$, exists and is not equal to $0$ or $1$ (we call this an $r$-way prime number race). Moreover, they have shown that if $r$ is fixed, then as $q\rightarrow \infty$,
$$ \max_{\substack{1\leq a_1,\dots,a_r \leq q \\ (a_i,q)=1}} \left|\delta(q;a_1,\dots,a_r)-\frac 1{r!}\right|\rightarrow0.$$ 
In other words, the bias dissolves as $q\rightarrow \infty$. For $r=2$, this phenomenon can readily seen in \cite{FiMa}, where the authors exhibit the list of the 117 densities which are greater than or equal to $9/10$. By the trivial inequality
$$ P_{q;a_1,\dots,a_r} \subset P_{q;a_1,a_2},  $$
we see that the most biased $r$-way prime number race is the two-way race appearing on top of the list in \cite{FiMa}, that is
$$ \delta(24;5,1)=0.999988...$$
Only one race is known to be more biased: it is the race between Li$(x)$ and $\pi(x)$, for which the density is $$\delta(1):=\delta(\{n:\text{Li}(n)>\pi(n)\})=0.99999973...$$

One can also combine different residue classes $\bmod q$ to make prime number races. For two subsets $A,B \subset (\Z/q\Z)^{\times}$, we consider the inequality
\begin{equation}
\frac 1{|A|} \sum_{a\in A} \pi(x;q,a) > \frac 1{|B|} \sum_{b\in B} \pi(x;q,b),
\label{equation race between A and B}
\end{equation} 
and denote by $\delta(q;A,B)$ the logarithmic density of the set of $x$ for which it is satisfied, if it exists.
An example of such race was given by Rubinstein and Sarnak who studied
the race between $$\pi(x;q,NR)=\#\{ p\leq x : p \text{ is not a quadratic residue} \bmod q\} $$ and $$\pi(x;q,R)=\#\{ p\leq x : p \text{ is a quadratic residue} \bmod q\},$$
for moduli $q$ having a primitive root. This race appears naturally in their work, since as they have shown, it is the property of the competitors being a quadratic residue or not which determines whether a two-way prime number race is biased or not. These are good candidates for biased races, however it can be shown that as $q\rightarrow \infty$, $\delta(q;NR,R)\rightarrow \frac 12$ (but at a much slower rate than two-way races, see \cite{FiMa}).

In general, one can see (\cite{BFHR}, \cite{FiMa}) that low-lying zeros (excluding real zeros) have a significant effect on decreasing the bias. However, real zeros have the reverse effect, and increase the bias. Nonetheless, real zeros are very rare, 
in fact Chowla's conjecture asserts that Dirichlet $L$-functions never vanish in the interval $s\in (0,1]$. 

Odlyzko \cite{Od} has shown that the Dirichlet $L$-function having the highest first zero in the critical strip is the Riemann zeta function, which is $\rho_0=\frac 12+i\cdot14.134725...$ Subsequently, Miller \cite{Mil} generalized this result by showing that each member of a very large class of cuspidal $GL_n$ $L$-functions has the property of either having a zero in the interval $[\frac 12-14.13472i,\frac 12+14.13472i]$, or having a zero whose real part is strictly larger than $1/2$ (violating GRH). In particular, this class contains all Dirichlet, rational elliptic curve and modular form $L$-functions, and possibly also contains all Artin and rational abelian variety $L$-functions.  
By these considerations, one might conjecture that the highest density one will ever find by doing prime number races is $\delta(1)=0.99999973...$ 

As it turns out, this is false, and we can find races which are arbitrarily biased. This is achieved by considering races between linear combinations of prime counting functions, and we will see in Section \ref{section general analysis} that the key to finding such biased races is to take a very large number of residue classes.

 The first (and most extreme) example we give is a quadratic residue versus quadratic non-residue race as in \cite{RubSar}, but for a general modulus $q$. 
We take $A=NR:=\{ a\bmod q : a \equiv  \square \bmod q\}$ and  $B=R:=\{ b\bmod q : b \not \equiv  \square \bmod q\}$ in \eqref{equation race between A and B}.
Note that $|B|=\phi(q)/\rho(q)$ and $|A|=\phi(q)\left(1-\frac 1{\rho(q)} \right)$, where
 $$ \rho(q):= [G:G^2]=\begin{cases}
2^{\omega(q)} &\text{if } 2\nmid q, \\
2^{\omega(q)-1} &\text{if } 2\mid q \text{ but } 4\nmid q, \\
2^{\omega(q)} &\text{if } 4\mid q \text{ but } 8\nmid q, \\
2^{\omega(q)+1} &\text{if } 8\mid q, \\
\end{cases}$$ 
and $\omega(q)$ denotes the number of distinct prime factors of $q$.

\begin{theorem}
Assume GRH and LI. Then for any $\epsilon>0$ there exists $q$ such that
\begin{equation} 1-\epsilon < \delta(q;NR,R)< 1. \label{delta tend vers 1} 
\end{equation}

Moreover, for any fixed $\frac 12 \leq \eta\leq 1$ there exists a sequence of moduli $\{q_n\}$ such that 
\begin{equation}  \lim_{n\rightarrow \infty} \delta(q_n;NR,R)= \eta. \label{delta tend vers eta}\end{equation}
 In concise form,
$$ \overline{ \{ \delta(q;NR,R) \} } = \Big[\frac 12,1\Big]. $$
\label{theorem extreme Dirichlet races}
\end{theorem}

To prove the existence of highly biased races we do not need the full strength of LI, in fact we only need a hypothesis on the multiplicity of the elements of the multiset of all non-trivial zeros of quadratic Dirichlet $L$-functions modulo $q$, which we will denote by $Z(q)$. Note that LI implies that the elements of this set have multiplicity one.

\begin{theorem}
Assume GRH, and assume that there exists an increasing sequence of moduli $q$ such that $\log q = o(\rho(q))$ and such that each element of $Z(q)$ has multiplicity $o(\rho(q)/\log q)$. Then for any $\epsilon>0$ there exists $q$ such that 
\begin{equation} 1-\epsilon < \underline{\delta}(q;NR,R) \leq \overline{\delta}(q;NR,R)< 1. \label{delta tend vers 1 sans LI} 
\end{equation}
\label{theorem extreme Dirichlet without LI}
\end{theorem}

\begin{remark}
The difference between \eqref{delta tend vers 1} and \eqref{delta tend vers 1 sans LI} is explained by the fact that it is not known whether $\delta(q;NR,R)$ exists under GRH alone.
\end{remark}

\begin{remark}
\label{remark consider squarefree moduli}
For a fixed modulus $q\geq 2$, write $q=2^e \prod_{\substack{p\mid q\\p\neq 2}} p^{e_p}$ and $\ell:=\prod_{\substack{p\mid q \\ p\neq 2}}p$. One can see that
$$ \delta(q;NR,R) = \delta(2^{\min(3,e)} \ell; NR,R), $$
since there are no real primitive characters modulo $p^e$ with $p\neq 2$ and $e\geq 2$, and there are no real primitive characters modulo $2^e$ for $e\geq 4$ (see Lemma \ref{lemma link with random variable X_q}). Therefore, when studying $\delta(q;NR,R)$ one can assume without loss of generality that $q$ is of the form $2^m \ell$, where $\ell$ is an odd squarefree integer and $m\leq 3$.
\end{remark}

\begin{remark}
\label{remark condition for biased races}
We will see that what controls the bias in these races is the number of prime factors of $q$ and the size of $q$. More precisely, under GRH and LI the two following statements are equivalent:
\begin{equation}
\sum_{p\mid q}\log p = o(2^{\omega(q)}),
\end{equation} 
 \begin{equation}
\delta(q;NR,R)=1-o(1).
\end{equation}
 Using this, we can show that the set of moduli $q\leq x$ such that $\delta(q;NR,R)=1-o(1)$ has density $(\log x)^{-\lambda+o(1)}$, where $\lambda=1-\frac{1+\log\log 2}{\log 2}=0.086071...$
Interestingly, Ford's work on integers having a divisor in a given interval (see \cite{Fo1}, \cite{Fo2}) shows that these integers appear in the Erd\H{o}s multiplication table.
\end{remark}

In terms of random variables, this can be explained by saying that the extreme examples we are considering correspond to random variables whose mean is much larger than their standard deviation. 
The easy way to show that this implies a very large bias is to use Chebyshev's inequality; however this approach is quite imprecise when the ratio 
$ \mathbb E[X]/\sqrt{\text{Var}[X]}$ is large. Instead, one should study the large deviations of $X-\mathbb E[X]$. The theory of large deviations of error terms arising from prime counting functions was initiated by Montgomery \cite{Mo}, and has since then been developed by Monach \cite{Mn}, Montgomery and Odlyzko \cite{MoOd}, Rubinstein and Sarnak \cite{RubSar}, and more recently Lamzouri \cite{Lam}. Exploiting such ideas we are able to be more precise in \eqref{delta tend vers 1}. 

\begin{theorem}
Assume GRH and LI, and define $q':=\prod_{p\mid q} p$. If $\rho(q)/\log q'$ is large enough, then we have
$$ \exp\left(-a_1 \frac{\rho(q)}{\log q'}\right) \leq 1-\delta(q;NR,R) \leq \exp\left(-a_2 \frac{\rho(q)}{\log q'}\right), $$
where $a_1$ and $a_2$ are absolute constants.
\label{theorem more precise large deviation}
\end{theorem}

This last theorem shows that the convergence in \eqref{delta tend vers 1} can be quite fast. It is actually possibly to explicitly compute a density which exceeds $\delta(1)$, namely $\delta(4849845;NR,R)=0.999999928...$ Below we list the first few values of $\delta(q;NR,R)$ for half-primorial moduli (that is, $q$ is the product of the first $k$ primes excluding $p=2$). These values were computed using Mysercough's method \cite{My} and Rubinstein's {\tt lcalc} package.

\begin{center}
\begin{tabular}{|c|c|c|c|}
\hline
$q$ & $\omega(q)$&$\rho(q)/\log q' $ & $\delta(q;NR,R)$\\
\hline
$3$ & $1$ &$1.82$ &$0.999063$ \\
$15$ & $2$ & $1.47$&$0.999907$ \\
$105$ & $3$& $1.71$ &$0.999928 $\\
$1155$ & $4$& $2.26$ & $0.999877$ \\
$15015$ & $5$& $3.33$ & $0.999950$ \\
$255255$ & $6$& $5.14$ & $0.9999946$\\
$4849845$ & $7$& $8.31$  & $0.999999928$ \\
$ 111546435$ & $8$ & $13.81$ & $0.999999999954$\\
\hline
\end{tabular}
\end{center}

\begin{remark}
As remarked by Rubinstein and Sarnak \cite{RubSar}, these densities can theoretically be computed to any given level of accuracy under GRH alone. Indeed, using the $B^2$ almost-periodicity of these races, this amounts to computing a finite number of zeros of Dirichlet $L$-functions to a certain level of accuracy. 
\end{remark}

\begin{remark}
One can summarize Remark \ref{remark condition for biased races}, Theorem \ref{theorem extreme Dirichlet races} and Theorem \ref{theorem more precise large deviation} by the following statement:
$$ \delta(q;NR,R) \approx \frac 1{\sqrt{2\pi}} \int_{-\sqrt{2^{\omega(q)-1}/\log q'}}^{\infty} e^{-\frac{x^2}2} dx. $$
\end{remark}

\begin{remark}
\label{remark pomerance}
Using our analysis, one can show that for almost all squarefree integers $q$,
$$ \delta(q;NR,R) -\frac 12 = (\log q)^{\frac{\log 2-1}2 +o(1)}. $$
That is to say, races with normal moduli have a very moderate bias.
\end{remark}

It is possible to analyse highly biased races in a more general setting, and to determine which features are needed for this bias to appear. To do this we take
$\a=(a_1,...,a_k)$ a vector of invertible reduced residues modulo $q$ and $\al=(\alpha_1,...,\alpha_k)$ a non-zero vector of real numbers such that $\sum_{i=1}^k\alpha_i=0$. We will be interested in the race between positive and negative entries of $\al$, that is we define
$$ \delta(q;\a,\al):= \delta(\{ n :\alpha_1\pi(n;q,a_1)+...+\alpha_k\pi(n;q,a_k) >0\}).$$
 Moreover, we define 
$$ \epsilon_i := \begin{cases}
1 & \text{ if }  a_i \equiv \square \bmod q\\
0 & \text{ if } a_i \not\equiv \square \bmod q,
\end{cases}
$$
and we assume without loss of generality that 
$$ \sum_{i=1}^k \epsilon_i \alpha_i < 0. $$
(By Lemma \ref{lemma link with random variables general}, this will force $\delta(q;\a,\al)> \frac 12$. If $ \sum_{i=1}^k \epsilon_i \alpha_i = 0,$ then $\delta(q;\a,\al)= \frac 12$. If $ \sum_{i=1}^k \epsilon_i \alpha_i> 0,$ then we multiply $\al$ by minus one and study the complementary probability $\delta(q;\a,-\al)=1-\delta(q;\a,\al)$.)

There are many choices of vectors $\a$ and $\al$ which yield highly biased races. We give some examples with constant coefficients, which we believe are the most natural.
\begin{theorem}
\label{theorem constant coefficients}
Assume GRH and LI, and let $k_R \leq \frac{\rho(q)}{\phi(q)}$ and $k_N \leq \left(1-\frac{1}{\rho(q)} \right)\phi(q)$ be two positive integers. Take $a_1,...,a_{k_N}$ to be any distinct quadratic non-residues $\bmod q$, with coefficients $\alpha_1=...=\alpha_{k_N}=k_R$, and $a_{k_N+1},...,a_{k_N+k_R}$ to be any distinct quadratic residues $\bmod q$, with $\alpha_{k_N+1}=...=\alpha_{k_N+k_R}=-k_N$. There exists an absolute constant $c>0$ such that if for some $0<\epsilon<\frac 1{2c}$ we have
\begin{equation}
\frac 1{k_N}+\frac 1{k_R} <\epsilon \frac{\rho(q)^2}{\phi(q)\log q},
\label{equation highly biased races with constant coefficients}
\end{equation} 
then 
$$ \delta(q;\a,\al) >1-c\epsilon.$$
\end{theorem}

\begin{remark}
Fix $0<\epsilon<\frac 1{2c}$ and define $N_{\epsilon}(q)$ to be to number of positive integers $k_N$, $k_R$ for which $k_N \leq (1-\frac 1{\rho(q)}) \phi(q)$, $k_R\leq \frac{\phi(q)}{\rho(q)}$, and 
$$ \frac 1{k_N}+\frac 1{k_R} < \epsilon  \frac{\rho(q)^2}{\phi(q)\log q}. $$
Then, for values of $q$ for which $\rho(q) \geq \epsilon^{-2} \log q$, we have that  $N_{\epsilon}(q)$ tends to infinity as $q\rightarrow \infty$. Hence, for values of $q$ for which $\log q =o(\rho(q))$, \eqref{equation highly biased races with constant coefficients} has a large number of solutions.
\end{remark}

\begin{remark}
\label{remark greg}
Theorem \ref{theorem constant coefficients} shows the existence of highly biased races with the same number of residue classes on each side of the inequality. Indeed, for moduli $q$ with $\log q=o(\rho(q))$, taking $k_R=k_N$ with $ \phi(q) \log q/\rho(q)^2=o(k_R)$ and choosing any residue classes $a_1,...,a_{k_N+k_R}$ gives a race with $\delta(q;\a,\al)=1-o(1)$.
\end{remark}

\begin{remark}
In Theorem \ref{theorem extreme Dirichlet races}, we have $k_N=\left(1-\frac 1{\rho(q)} \right) \phi(q)$ and $k_R= \frac{\phi(q)}{\rho(q)}$, which explains why we obtained a highly biased race when $\rho(q)$ was large compared to $\log q$.
\end{remark}

Here is our most general class of highly biased races.
\begin{theorem}
\label{theorem general way of being biased}
Assume GRH and LI. There exists an absolute constant $c>0$ such that if for some $0<\epsilon<\frac 1{2c}$ we have
\begin{equation}
\frac{\sum_{i=1}^k\alpha_i^2}{\left( \sum_{i=1}^k \epsilon_i\alpha_i\right)^2} < \epsilon \frac{\rho(q)^2}{\phi(q)\log q},
\label{equation condition pour grosses courses generales}
\end{equation}
then
$$ \delta(q;\a,\al) >1-c\epsilon.$$
\end{theorem}
\begin{remark}
Trivially, one has 
$$ \frac{\sum_{i=1}^k\alpha_i^2}{\left( \sum_{i=1}^k \epsilon_i\alpha_i\right)^2} \geq \frac 1{k_R}, $$
where $k_R:=\sum_{i=1}^k\epsilon_i$.
Hence, for \eqref{equation condition pour grosses courses generales} to be satisfied, one needs $k_R$ to be larger than 
$$\epsilon^{-1}\frac{\phi(q)\log q}{\rho(q)^2}.$$
Since $k_R\leq \frac{\phi(q)}{\rho(q)}$, this imposes the following condition on $q$:
$$ \rho(q) \geq \epsilon^{-1}\log q. $$
\end{remark}
\begin{remark}
The goal of Theorem \ref{theorem general way of being biased} is to give a large class of biased races, without necessarily being precise on the value of $\delta(q;\a,\al)$. One can use the Montgomery-Odlyzko bounds \cite{MoOd} to obtain more precise estimates in some particular cases.
\end{remark}

The previous examples of highly biased races all have the property that the number of residue classes involved is very large in terms of $q$ (it is at least $q^{1-o(1)}$). In the next theorem we show that this condition is necessary, and that moreover highly biased are very particular, in the sense that they must satisfy precise conditions.

\begin{theorem}
\label{theorem limitations generales}
Assume GRH and LI. There exists absolute positive constants $K_1,K_2$ and $0<\eta<1/2$ such that if $k\leq K_1 \phi(q)$ and
\begin{equation}
\frac{\left( \sum_{i=1}^k \epsilon_i\alpha_i\right)^2}{\sum_{i=1}^k\alpha_i^2} \leq K_2 \frac{\phi(q) \log (3\phi(q)/k)}{\rho(q)^2},  
\label{equation hypothesis thm not biased}
\end{equation} 
then
\begin{equation}
 \delta(q;\a,\al)\leq 1-\eta.
 \label{equation conclusion of theorem not biased}
\end{equation}
(Hence this race cannot be too biased.)
\end{theorem}

\begin{remark}
Applying the Cauchy-Schwartz and using that $k_R:=\sum_{i=1}^k \epsilon_i \leq \phi(q)/\rho(q)$, one sees that if $\rho(q)\leq K_2 \log \left(3\phi(q)/k \right)$, then whatever $\a$ and $\al$ are, \eqref{equation hypothesis thm not biased} holds. Moreover, in the range $\rho(q)> K_2 \log \left(3\phi(q)/k \right)$ we have that if $k_R \leq K_2 \phi(q)/\rho(q)^2$, then \eqref{equation hypothesis thm not biased} holds. We conclude that a necessary condition to obtain a highly biased race is that $k_R\gg \phi(q)/\rho(q)^2$.
\end{remark}

An interesting feature of prime number races is Skewes' number. It is by definition the smallest $x_0$ for which
$$ \pi(x_0)> Li(x_0).$$
This number has been extensively studied since Skewes' 1933 paper in which he showed under GRH that 
$$ x_0<10^{10^{10^{34}}}.$$
The GRH assumption has since then be removed and the upper bound greatly reduced; we refer the reader to \cite{BaHu} for the list of such improvements. The current record is due to Bays and Hudson \cite{BaHu}, who showed that  $x_0<1.3983 \times 10^{316}$, and moreover this bound is believed to be close to the true size of $x_0$. 

One could also study the generalized Skewes' number
$$ x_{q;a,b} := \inf \{ x : \pi(x;q,a)<\pi(x;q,b) \}. $$ 
However, two-way prime number races become less and less biased as $q$ grows, that is $\delta(q;a,b) \rightarrow \frac 12$ uniformly in $a$ and $b$ coprime to $q$. Hence, for large $q$ we expect this generalized Skewes number to be small and uninteresting.

The situation is quite different with the highly biased we constructed, in fact we expect the Skewes number
$$ x_{q} := \inf \{ x : (\rho(q)-1) \pi(x;q,NR)<\pi(x;q,R) \}$$
to tend to infinity as $\rho(q)/\log q'$ tends to infinity ($q'$ is the radical of $q$). One can then ask the following question: how fast does it tend to infinity? Similar arguments to those of Montgomery \cite{Mo} and of Ng \cite{Ng} allow us to make the speculation that the answer is double-exponentially.

\begin{conjecture}
As $\rho(q)/\log q'$ tends to infinity we have
$$ \log\log x_q \asymp \frac{\rho(q)}{\log q'}. $$
\end{conjecture} 

\section*{Acknowledgements}
I would like to thank my former advisor Andrew Granville for his very interesting question which motivated this work, and for his comments and encouragement. I would also like to thank Enrico Bombieri and Peter Sarnak for their advice and their encouragement, and the Institute for Advanced Study for providing excellent research conditions. I thank Kevin Ford, Youness Lamzouri and the people at the University of Illinois at Urbana-Champaign for their hospitality and for very fruitful conversations which led me to consider general linear combinations of prime counting functions. I thank Jan-Christoph Schlage-Puchta for Remark \ref{remark schlage} and for suggesting the current proof of Lemma \ref{lemma moments in terms of E(y)}. I also thank Carl Pomerance for suggesting Remark \ref{remark pomerance}, and Greg Martin for suggesting Remark \ref{remark greg}. I thank Barry Mazur for motivating me to weaken the Linear Independence Hypothesis.
Finally, I thank Byungchul Cha, Ke Gong, Chen Meiri, Nathan Ng and Anders Södergren for helpful conversations. This work was supported by an NSERC Postdoctoral Fellowship, as well as NSF grant DMS-0635607.

\section{Results without the linear independence hypothesis}
\label{section without LI}

The goal of this section is to prove Theorem \ref{theorem extreme Dirichlet without LI} (from which the first part of Theorem \ref{theorem extreme Dirichlet races} clearly follows). We first note that if $A=NR$ and $B=R$, then \eqref{equation race between A and B} is equivalent to
$$\pi(x;q,NR)>(\rho(q)-1)\pi(x;q,R).$$

\begin{lemma}
\label{lemma error term}
Assuming GRH, we have that
 $$ E_q(x):=\frac{\pi(x;q,NR)-(\rho(q)-1)\pi(x;q,R)}{\sqrt x/\log x} =\rho(q)-1+\sumchi \sum_{\gamma_{\chi}}  \frac{x^{i\gamma_{\chi}}}{\rho_{\chi}}+o(1). $$
\end{lemma}

\begin{proof}

Let $b$ be an invertible reduced residue $\bmod q$. We will use the orthogonality relation
\begin{equation}
\sum_{\substack{ \chi\bmod q \\ \chi^2=\chi_0 \\ \chi\neq \chi_0}} \chi(b) = \begin{cases}
\rho(q)-1 &\text{ if } b\equiv \square \bmod q \\
-1 & \text{ if } b \not\equiv \square \bmod q.
\end{cases}
\label{equation orthogonality relations residues non residues}
\end{equation} 
The explicit formula gives
\begin{equation}
 \sum_{\substack{ \chi\bmod q \\ \chi^2=\chi_0 \\ \chi\neq \chi_0}} \psi(x,\chi) =- \sum_{\substack{ \chi\bmod q \\ \chi^2=\chi_0 \\ \chi\neq \chi_0}}  \sum_{\rho_{\chi}} \frac{x^{\rho_{\chi}}}{\rho_{\chi}}+O_q(\log x),
 \label{equation explicit formula}
\end{equation}
where $\rho_{\chi}$ runs over the non-trivial zeros of $L(s,\chi)$. The left hand side of \eqref{equation explicit formula} is equal to
\begin{align*}&\sum_{\substack{ \chi\bmod q \\ \chi^2=\chi_0 \\ \chi\neq \chi_0}} \sum_{p\leq x} \chi(p) \log p  + \sum_{\substack{ \chi\bmod q \\ \chi^2=\chi_0 \\ \chi\neq \chi_0}} \sum_{p^2\leq x}\chi(p)^2 \log p +O(x^{\frac 13}) \\ &= (\rho(q)-1)\sum_{\substack{p\leq x \\ p  \equiv \square \bmod q}} \log p - \sum_{\substack{p\leq x \\ p  \not\equiv \square \bmod q}} \log p +(\rho(q)-1)\sqrt x +o(\sqrt x),
\end{align*}
by \eqref{equation orthogonality relations residues non residues} and the Prime Number Theorem. Combining this with a standard summation by parts we get that
 $$ \frac{\pi(x;q,NR)-(\rho(q)-1)\pi(x;q,R)}{\sqrt x/\log x} =\rho(q)-1+\sum_{\substack{ \chi\bmod q \\ \chi^2=\chi_0 \\ \chi\neq \chi_0}} \sum_{\gamma_{\chi}}  \frac{x^{i\gamma_{\chi}}}{\rho_{\chi}}+o(1). $$
\end{proof}

\begin{lemma}
The quantity $E_q(x)$ defined in Lemma \ref{lemma error term} has a limiting logarithmic distribution, that is there exists a Borel measure $\mu_q$ on $\R$ such that for any bounded Lipschitz continuous function $f:\R \rightarrow \R$ we have
$$ \lim_{Y\rightarrow \infty} \frac 1Y \int_2^{Y} f(E_q(e^y)) dy = \int_{\mathbb R} f(t) d\mu_q(t).$$
\label{lemma limiting distribution}
\end{lemma}
\begin{proof}
This follows from Rubinstein and Sarnak's anaysis \cite{RubSar}, and from \cite{NgSh}.
\end{proof}

\begin{remark}
\label{remark schlage}
As Schlage-Puchta has pointed out to me, it is possible to show under GRH that for all but a countable set of values of $c$, the density
$$ F_q(c):= \lim_{Y\rightarrow \infty} \frac 1Y meas\{ y\leq Y : E_q(e^y) \leq c \} $$ 
exists. Moreover, one can show that in the domain where $F$ is defined,
\begin{multline*} \sup_{x<c} F_q(x) \leq \liminf_{Y\rightarrow \infty} \frac 1Y meas\{ y\leq Y : E_q(e^y) \leq c \}  \\ \leq \limsup_{Y\rightarrow \infty} \frac 1Y meas\{ y\leq Y : E_q(e^y) \leq c \} \leq  \inf_{x>c} F_q(x), 
\end{multline*}
and so in particular if $F_q(x)$ is continuous at $x=c$, then the set $\{ y\leq Y : E_q(e^y) \leq c \} $ has a density.

\end{remark}

Let $X_q$ be the random variable associated to $\mu_q$. We will show that $X_q$ can be very biased, in the sense that $\P[X_q>0]$ can be very close to $1$. To do so we will compute the first two moments of $E_q(e^y)$, which we relate to the random variable $X_q$.

\begin{lemma}
We have that 
$$ \lim_{Y\rightarrow \infty} \frac 1Y \int_2^{Y} E_q(e^y) dy = \int_{\mathbb R} t d\mu_q(t), $$
$$ \lim_{Y\rightarrow \infty} \frac 1Y \int_2^{Y} E_q(e^y)^2 dy = \int_{\mathbb R} t^2 d\mu_q(t). $$
\label{lemma moments in terms of E(y)}
\end{lemma}

\begin{proof}
We will only prove the second statement, as the first follows along the same lines.
Similarly as in \cite{Pu}, we can compute that
$$ \lim_{Y\rightarrow \infty} \frac 1Y \int_0^Y |E_q(e^y)|^4 dy = \sum_{\rho_1+\rho_2+\rho_3+\rho_4 =0} \frac 1{ \rho_1\rho_2\rho_3\rho_4 } <\infty, $$
where the last sum runs over quadruples of non-trivial zeros of quadratic Dirichlet $L$-functions modulo $q$. This implies that as $M\rightarrow \infty$,
\begin{equation}
\limsup_{Y\rightarrow \infty} \frac 1Y \int_{\substack{0\leq y \leq Y \\ |E_q(e^y)|>M }} |E_q(e^y)|^2 dy \longrightarrow 0.
\label{equation second moment}
\end{equation}  
Indeed, if this was not the case then we would have that for all $M>M_0$,
$$  \limsup_{Y\rightarrow \infty} \frac 1Y \int_{\substack{0\leq y \leq Y \\ |E_q(e^y)|>M }} |E_q(e^y)|^2 dy \geq \eta>0,  $$
and so
$$ \limsup_{Y\rightarrow \infty} \frac 1Y \int_{\substack{0\leq y \leq Y \\ |E_q(e^y)|>M }} |E_q(e^y)|^4 dy \geq \eta M^2, $$
which would contradict the fact that the fourth moment is finite. We now define the bounded Lipschitz function
$$ H_M(t):= \begin{cases}
t^2 &\text{ if } |t|\leq M \\
M^2(M+1-|t|) &\text{ if } M< |t| \leq M+1 \\
0 &\text{ if } |t|\geq M+1.
\end{cases} $$
We then have

\begin{multline*} \frac 1Y \int_2^{Y} E_q(e^y)^2 dy  = \frac 1Y \int_{ 2\leq y\leq Y } H_M(E_q(e^y)) dy -   \frac 1Y \int_{ \substack{2\leq y\leq Y \\ M<|E_q(e^y)| \leq M+1 }} H_M(E_q(e^y)) dy \\+  \frac 1Y \int_{ \substack{2\leq y\leq Y \\ |E_q(e^y)| >M}} E_q(e^y)^2 dy, 
\end{multline*}
therefore by \eqref{equation second moment} and by Lemma \ref{lemma limiting distribution} we get that
$$ \limsup_{ Y\rightarrow \infty} \frac 1Y \int_2^{Y} E_q(e^y)^2 dy =  \int_{ \mathbb R } H_M(t) d\mu_q(t) + \epsilon_M, $$
where $\epsilon_M$ tends to zero as $M\rightarrow \infty$. Using the bound
$$\mu_q( (-\infty,-M]\cup [M,\infty) ) \ll \exp(-c_2 \sqrt M) $$
(see Theorem 1.2 of \cite{RubSar}) we get by taking $M\rightarrow \infty$ that
$$ \limsup_{ Y\rightarrow \infty} \frac 1Y \int_2^{Y} E_q(e^y)^2 dy =  \int_{ \mathbb R } t^2 d\mu_q(t). $$
The same reasoning applies to the $\liminf$, and thus the proof is finished.

\end{proof}

The following calculation is similar to that of Schlage-Puchta \cite{Pu}, who computed the moments of $e^{-t/2}\psi(e^t;\chi)$.
\begin{lemma}
Assume GRH. Then,
$$ \E[X_q]= \rho(q) -1 +z(q), \hspace{1cm} \V[X_q] = \sum^*_{\gamma\neq 0} \frac{m_{\gamma}^2}{\frac 14+\gamma^2},$$
where the last sum runs over the imaginary parts of the non-trivial zeros of 
$$ Z_q(s) := \prod_{\substack{\chi^2=\chi_0 \\ \chi \neq \chi_0}} L(s,\chi),$$
$m_{\gamma}$ denotes the multiplicity of the zero $\frac 12+i\gamma$, the star meaning that we count the zeros without multiplicity, and $z(q)$ denotes the multiplicity of the (possible) real zero $\gamma=0$.
\label{lemma first two moments without LI}
\end{lemma}
\begin{proof}
By Lemma \ref{lemma error term} we have that 
\begin{align*}\int_2^{Y} E_q(e^y) dy &= (\rho(q)-1+z(q))(Y-2) + \sumchi \sum_{ \gamma_{\chi} \neq 0} \frac 1{ \frac 12+i\gamma_{\chi}} \int_2^{Y} e^{i\gamma_{\chi} y} dy +o_{Y\rightarrow \infty}(Y)  \\
&= (\rho(q)-1+z(q))(Y-2) + O_q\left(  1 \right)+o_{Y\rightarrow \infty}(Y),
\end{align*}
by absolute convergence.
Taking $Y \rightarrow \infty$ and applying Lemma \ref{lemma moments in terms of E(y)} gives that
$$\E[X_q] = \lim_{Y\rightarrow \infty} \frac 1Y \int_2^{Y} E_q(e^y) dy= \rho(q)-1+z(q). $$

The calculation of the variance follows from Lemma \ref{lemma error term} and from Parseval's identity for $B^2$ almost-periodic functions \cite{Be1}. (An alternative way to compute the variance is to argue as in \cite{Pu}.)
\end{proof}

\begin{remark}
It is a general fact that Besicovitch almost-periodic functions always have a mean value \cite{Be2}. Moreover, Parseval's identity \cite{Be1,Be2} shows that Besicovitch $B^2$ (and thus also Stepanov $S^2$, Weil $W^2$ and Bohr) 
almost periodic functions $f(y)$ have a second moment given by 
$$  \lim_{Y \rightarrow \infty } \frac 1Y \int_0^Y f(y)^2 dy = \sum_{n\geq 1} A_n^2, $$
where the $A_n$ are the Fourier coefficients of $f$. 
\end{remark}

\begin{lemma}
Assume GRH. If $$B(q):=\frac{\E[X_q]}{\sqrt{\V[X_q]}}$$
is large enough, then 
$$  \underline{\delta}(q;NR,R) \geq 1-2\frac{\V[X_q]}{\E[X_q]^2} . $$
\label{lemma bias factor chebyshev}
\end{lemma}
\begin{proof}
It is clear from Lemma \ref{lemma first two moments without LI} and the Riemann-von Mangoldt formula that 
$\V[X_q] \gg \log q'$, and therefore our assumption that $B(q)$ is large enough implies that $\E[X_q]$ is also large enough, say at least $4$. 
Define
$$H(x):= \begin{cases} 0  &\text{ if } x <0 \\ 1 &\text{ if } x\geq 0
\end{cases} \hspace{1cm}
f(x):= \begin{cases} 0 &\text{ if } x <0 \\ 
x &\text{ if } 0\leq x < 1 \\
1 &\text{ if } x\geq 0. 
\end{cases}
 $$ 
Clearly, $f(x)$ is bounded Lipschitz continuous and $f(x)\leq H(x)$. Therefore,

$$ \underline{\delta}(q;NR,R) = \liminf_{Y\rightarrow \infty} \frac 1Y \int_2^Y H(E_q(e^y)) dy \geq  \liminf_{Y\rightarrow \infty} \frac 1Y \int_2^Y f(E_q(e^y)) dy, $$
which by Lemma \ref{lemma limiting distribution} is equal to
\begin{align*} \int_{\mathbb R} f(t) d\mu_q(t)&= 1-\int_{\mathbb R} (1-f(t)) d\mu_q(t)\\
&=1-\int_{-\infty}^{1} (1-f(t)) d\mu_q(t)\geq 1-\mu_q(-\infty,1]. 
\end{align*}

We now apply Chebyshev's inequality:
\begin{multline*}
\mu_q(-\infty,1] = \P[X_q \leq 1]  = \P[X_q-\E[X_q] \leq 1-\E[X_q]]\\ \leq \P[|X_q-\E[X_q]| \geq \E[X_q]-1] 
 \leq \frac{\V[X_q]}{(\E[X_q]-1)^2} \leq 2\frac{\V[X_q]}{\E[X_q]^2} 
\end{multline*}
since $\E[X_q]\geq 4$, and therefore 
$$\underline{\delta}(q;NR,R) \geq 1-2\frac{\V[X_q]}{\E[X_q]^2}. $$
\end{proof}

\begin{proof}[Proof of Theorem \ref{theorem extreme Dirichlet without LI}]
By Lemma \ref{lemma first two moments without LI}, our hypothesis implies that for the sequence of moduli $q$ under consideration, 
$$ \V[X_q] \leq \max_{\gamma}(m_{\gamma})  \sum_{\gamma}^* \frac{m_{\gamma}}{\frac 14+\gamma_{\chi}^2}  = o\left( \frac{\rho(q)}{\log q} \rho(q) \log q \right) = o(\rho(q)^2), $$
by the Riemann von-Mangoldt formula. Lemma \ref{lemma first two moments without LI} also implies that $\E[X_q] \gg \rho(q)$, and hence Lemma \ref{lemma bias factor chebyshev} implies that 
$$ \underline{\delta}(q;NR,R) \geq 1-o(1).$$
The last inequality to show, that is $ \overline{\delta}(q;NR,R) <1$, follows from an analysis using the functions $f(x)$ and $H(x)$ of Lemma \ref{lemma bias factor chebyshev}, combined with a lower bound on $\mu_E(-\infty,-1]$ similar to that in Theorem 1.2 of \cite{RubSar}, which holds in greater generality \cite{NgSh}.

\end{proof}

\section{A central limit theorem}
The goal of this section is to show a central limit theorem under GRH and LI, from which the second part of Theorem \ref{theorem extreme Dirichlet races} will follow. 
We first translate our problem to questions on sums of independent random variables, which can be done thanks to hypothesis LI. Recall that we are interested in the set of $n$ such that
$$ \pi(n;q,NR)>(\rho(q)-1)\pi(n;q,R).$$

\begin{lemma}
\label{lemma link with random variable X_q}
Assume GRH and LI. Then the logarithmic density of the set of $n$ for which
$ \pi(n;q,NR)>(\rho(q)-1)\pi(n;q,R)$ exists and equals 
$$ \P[X_q]>0,$$
where $X_q$ is the random variable defined in Section \ref{section without LI}. Moreover we have
\begin{equation}
X_q\sim\rho(q)-1+\sum_{\substack{ \chi\bmod q \\ \chi^2=\chi_0 \\ \chi\neq \chi_0}} \sum_{\gamma_{\chi}>0} \frac{2\Re(Z_{\gamma_{\chi}})}{\sqrt{\frac 14+\gamma_{\chi}^2}},
\label{equation residues versus non residues random variable}
\end{equation}
where the $Z_{\gamma_{\chi}}$ are independent identically distributed random variables following a uniform distribution on the unit circle in $\mathbb C$.

\end{lemma}
\begin{proof}

By Lemma \ref{lemma error term}, we have that
 $$ \frac{\pi(x;q,NR)-(\rho(q)-1)\pi(x;q,R)}{\sqrt x/\log x} =\rho(q)-1+\sum_{\substack{ \chi\bmod q \\ \chi^2=\chi_0 \\ \chi\neq \chi_0}} \sum_{\gamma_{\chi}}  \frac{x^{i\gamma_{\chi}}}{\rho_{\chi}}+o(1), $$
since LI implies that there are no real zeros.
 It follows by the work of Rubinstein and Sarnak that $\delta(q;NR,R)$ exists and equals $\text{Prob}[X_q>0]$ (their analysis shows that the distribution function of $X_q$ is continuous). Moreover, an argument similar to the proof of Proposition 2.3 of \cite{FiMa} shows that \eqref{equation residues versus non residues random variable} holds.
\end{proof}

One can show that the random variables appearing in \eqref{equation residues versus non residues random variable} have variance $\V[\Re(Z_{\gamma_{\chi}})]=\frac 12$, and have mean $\E[Z_{\gamma_{\chi}}]=0$. Using this and the fact that they are mutually independent, we recover Lemma \ref{lemma first two moments without LI}:
\begin{equation}
\mathbb E[X_q]=\rho(q)-1, \hspace{1cm} \text{Var}[X_q] = \sum_{\substack{ \chi\bmod q \\ \chi^2=\chi_0 \\ \chi\neq \chi_0}} \sum_{\gamma_{\chi}} \frac{1}{\frac 14+\gamma_{\chi}^2}, 
\label{equation first two moments}
\end{equation} 
since the zeros come in conjugate pairs ($\chi$ is real).
We will see in the following lemma that $\text{Var}[X_q] \asymp \rho(q) \log q'$ (recall that $q':=\prod_{p\mid q}p$), and this is a crucial fact in our analysis. 

\begin{lemma}
\label{lemma variance of X_q}
Assume GRH and let $X_q$ be the random variable defined in \eqref{equation residues versus non residues random variable}. We have that
$$ \V[X_q] = 2^{\omega(q)-1-\epsilon_q} \log q' \left[ 1+O\left( \frac{\log\log q'}{\log q'}\right)\right],$$
where $\epsilon_q=1$ if $2\mid q$, and $\epsilon_q=0$ otherwise. In particular, 
$$ \V[X_q] \asymp \rho(q) \log q'.$$
\end{lemma}
\begin{proof}
By Remark \ref{remark consider squarefree moduli}, we have that 
$$ \V[X_q]=\V[X_{2^e\ell}], $$
where $e\leq 3$, $2^e \parallel q$ and $\ell := \prod_{\substack{p\mid q \\ p\neq 2}}p$. Therefore we assume from now on (without loss of generality) that $q=2^e\ell$, with $e\leq 3$ and $\ell$ an odd squarefree integer.

 Lemma 3.5 of \cite{FiMa} gives that
\begin{align}
\label{equation sum over zeros B(chi)}
\begin{split}\sum_{\gamma_{\chi}} \frac{1}{\frac 14+\gamma_{\chi}^2} &= \log q^* -\log \pi -\gamma -(1+\chi(-1))\log 2 +2\Re \frac{L'}L(1,\chi^*) \\
&= \log q^* + O(\log\log q^*), 
\end{split}
\end{align}
by Littlewood's GRH bound on $\frac{L'}L(1,\chi)$. Plugging this into \eqref{equation first two moments} we get
$$ \V[X_q]=\sum_{\substack{ \chi\bmod q \\ \chi^2=\chi_0 }} \log q^* + O(2^{\omega(q)}\log\log q).$$
If $q$ is odd, then there is exactly one primitive real character $\bmod d$ for every $d\mid q$, hence
$$\sum_{\substack{ \chi\bmod q \\ \chi^2=\chi_0 }} \log q^* = \sum_{d\mid q} \log d = \sum_{d\mid q} \sum_{p\mid d} \log p= \sum_{p\mid q} \log p 2^{\omega(q)-1} =  2^{\omega(q)-1} \log q. $$ 
If $2\parallel q$, then there are no primitive characters modulo even divisors of $q$, so
$$\sum_{\substack{ \chi\bmod q \\ \chi^2=\chi_0 }} \log q^* = \sum_{d\mid \frac q2 } \log d  =  2^{\omega(q)-2} \log \frac q2. $$ 
If $4\parallel q$, then there is exactly one primitive real character modulo divisors which are a multiple of $4$, so
$$\sum_{\substack{ \chi\bmod q \\ \chi^2=\chi_0 }} \log q^* = \sum_{d\mid \frac q4 } \log d +\sum_{4\mid d\mid  q } \log d = 2^{\omega(q)-2}\log (2q). $$
If $8\parallel q$, then there are exactly two primitive real characters modulo divisors which are a multiple of $8$, so
$$\sum_{\substack{ \chi\bmod q \\ \chi^2=\chi_0 }} \log q^* = \sum_{d\mid \frac q8 } \log d +\sum_{\substack{4\mid d\mid  q \\ 8\nmid d }} \log d+2\sum_{8\mid d\mid  q } \log d = 2^{\omega(q)-2}\log (8q). $$
\end{proof}

Let $X_q$ be the random variable defined in \eqref{equation residues versus non residues random variable}, and define
\begin{equation*}B(q):= \frac{\mathbb E[X_q]}{\sqrt{\text{Var}[X_q]}} . \label{definition of K_q}
\end{equation*}
It is $B(q)$ which dictates the behaviour of the race we are considering: if $B(q)$ is small, then the race will not be very biased, whereas if $B(q)$ is large, then the race will have a significant bias. By Lemma \ref{lemma variance of X_q}, we have the estimate
\begin{equation}
B(q)= \sqrt{\frac{2^{\omega(q)+1+\epsilon_q}}{\log q'}}\left[1+O\left(2^{-\omega(q)}+\frac{\log \log q'}{\log q'}\right)\right].
\label{equation computation of B(q)}
\end{equation}  

To prove the second part of Theorem \ref{theorem extreme Dirichlet races} we will need a sequence of moduli for which $B(q)$ is very regular.

\begin{lemma}
\label{lemma prescribed prime factors}
For any fixed $0<c<\infty$, there exists an increasing sequence of squarefree odd integers $\{ q_n \}$ such that 
$$ 2^{\omega(q_n)+1} = (c+o(1)) \log q_n. $$
\end{lemma}
\begin{proof}
Fix $0<c<\infty$, and define $e_c:= \min\{e\geq 1 : 2^{-e}c<\frac 2{\log 4}\}$ and $c_1:=2^{-e_c}c<\frac 2{\log 4}$. Define for $\ell=1,2,...$ the intervals
$$ I_{\ell}:=(\exp(c_1^{-1}2^{\ell}),2\exp(c_1^{-1}2^{\ell})), \hspace{1cm} J_{\ell}:=(2\exp(c_1^{-1}2^{\ell}),4\exp(c_1^{-1}2^{\ell})). $$
Since $c_1<\frac{2}{\log 4}$ we have that for all $\ell\geq1$,
$$4\exp(c_1^{-1}2^{\ell}) < \exp(c_1^{-1}2^{\ell+1}); $$
hence our intervals are all disjoint. We define $p_{\ell}$ to be any prime in the interval $I_{\ell}$, and similarly for $p'_{\ell}\in J_{\ell}$. The existence of such primes is granted by Bertrand's postulate (note that $\exp(c_1^{-1}2^1)>4$). Now, the sequence of moduli we are looking for is
$$ q_n:= \prod_{1\leq \ell \leq e_c} p'_{\ell} \prod_{1\leq \ell \leq n} p_{\ell}, $$
since
\begin{align*}
\frac{2^{\omega(q_n)+1}}{\log q_n} &= \frac{2^{n+e_c+1}}{O_c(1)+ \sum_{1\leq \ell\leq n} (c_1^{-1}2^{\ell} +O(1) )}=\frac{2^{n+e_c+1}}{ c_1^{-1}2^{n+1} +O_c(n)} \\&= 2^{e_c}c_1\left(1+O_c\left( \frac n{2^n}\right) \right)=c(1+o(1)),
\end{align*}
by definition of $c_1$.

\end{proof}

Before proving the second part of Theorem \ref{theorem extreme Dirichlet races}, we give some information about the characteristic function of the random variables we are interested in. The following lemma implies a central limit theorem.

\begin{lemma}
Let $X_q$ be the random variable defined in \eqref{equation residues versus non residues random variable}, and define
$$Y_q:=\frac {X_q-\E[X_q]}{\sqrt{\V[X_q]}}=\frac 1{\sqrt{\V[X_q]}}\sum_{\substack{ \chi\bmod q \\ \chi^2=\chi_0 \\ \chi\neq \chi_0}} \sum_{\gamma_{\chi}>0} \frac{2\Re(Z_{\gamma_{\chi}})}{\sqrt{\frac 14+\gamma_{\chi}^2}}.$$
The characteristic function of $Y_q$ satisfies, for $|\xi|\leq \frac 35 \sqrt{\V[X_q]}$,
$$ \hat Y_q(\xi)= -\frac{\xi^2}2+O\left( \frac{\xi^4}{\rho(q)\log q'} \right). $$
Moreover, in the same range we have
\begin{equation}
 \hat Y_q(\xi)\leq -\frac{\xi^2}2. 
 \label{equation bound on Y_q}
\end{equation}
\label{lemma clt}
\end{lemma}

\begin{proof}
The proof is very similar to that of Theorem 3.22 of \cite{FiMa}. Using the additivity of the cumulant-generating function of $X_q$, one can show that 
\begin{equation}
\log \hat X_q(\xi) =i \E[X_q] \xi+ \sum_{\substack{ \chi\bmod q \\ \chi^2=\chi_0 \\ \chi\neq \chi_0}} \sum_{\gamma_{\chi}>0} \log \left(J_0 \left( \frac{2 \xi}{\sqrt{\frac 14+\gamma_{\chi}^2}} \right) \right). 
\label{equation characteristic function of X_q}
\end{equation} 
We will use the following Taylor expansion, which is valid for $|\xi|\leq \frac {12}5$ (see Section 2.2 of \cite{FiMa}):
\begin{equation}
 \log J_0(\xi) = -\frac{\xi^2}4 +O(\xi^4).
 \label{equation Taylor expansion of Bessel}
\end{equation}
Plugging this estimate into \eqref{equation characteristic function of X_q} we get that for $|\xi|\leq \frac 35$,
$$ \log \hat X_q(\xi) = i \E[X_q] \xi-\xi^2 \sum_{\substack{ \chi\bmod q \\ \chi^2=\chi_0 \\ \chi\neq \chi_0}} \sum_{\gamma_{\chi}>0}  \frac{1}{\frac 14+\gamma^2} +O \left( \xi^4 \sum_{\substack{ \chi\bmod q \\ \chi^2=\chi_0 \\ \chi\neq \chi_0}}\sum_{\gamma_{\chi}>0}  \frac{1}{\left(\frac 14+\gamma^2\right)^2}  \right).  $$
One can show using the Riemann-von Mangoldt formula that $$ \sum_{\substack{ \chi\bmod q \\ \chi^2=\chi_0 \\ \chi\neq \chi_0}} \sum_{\gamma_{\chi}>0}  \frac{1}{\left(\frac 14+\gamma^2\right)^2} \ll \rho(q) \log q'.$$
Moreover, by Lemma \ref{lemma variance of X_q} we have $\V[X_q]\asymp \rho(q)\log q'$. Putting these together and using \eqref{equation first two moments}, we get that
$$ \log \hat Y_q(\xi) =\log \hat X_q\left(\frac{\xi}{\sqrt{\V[X_q]}}\right) - i\E[X_q] \frac{\xi}{\sqrt{\V[X_q]}} = -\frac{\xi^2}2 + O\left( \frac{\xi^4}{\rho(q)\log q'}\right), $$
showing the first assertion. For the second we use the same argument, but we replace the estimate \eqref{equation Taylor expansion of Bessel} with the following inequality, valid in the range $|\xi|\leq \frac {12}5$:
$$ \log J_0(\xi) \leq -\frac{\xi^2}4. $$
\end{proof}

\begin{lemma}[Berry-Essen inequality]
Denote by $F_q$ the distribution function of 
$$Y_q:=\frac {X_q-\E[X_q]}{\sqrt{\V[X_q]}},$$ and by $F$ that of the Gaussian distribution. We have that
$$ \sup_{x \in \mathbb R}|F_q(x)-F(x)| \ll \frac 1{\rho(q)\log q'}. $$
\label{lemma berry esseen}
\end{lemma}
\begin{remark}
One could get a more precise estimate using the Martin-Feuerverger formula \cite{FeMa}. However, the estimate of Lemma \ref{lemma berry esseen} is sufficient for our purposes. 
\end{remark}
\begin{proof}
Since the statement is trivial if $\rho(q)\log q'$ is bounded, we can assume without loss of generality that $\V[X_q]\geq 1$ (by Lemma \ref{lemma variance of X_q}).

The Berry-Esseen inequality in the form given by Esseen (Theorem 2a of \cite{Ess}) gives that for any $T>0$,
\begin{equation}
\sup_{x \in \mathbb R}|F_q(x)-F(x)|  \ll \int_{-T}^T \frac{\hat Y_q(\xi)-e^{-\frac{\xi^2}2}}{\xi}d\xi + \frac 1T.
\label{equation Berry-Esseen}
\end{equation} 
We take $T:=\V[X_q]$. By Lemma \ref{lemma clt}, the part of the integral with $|\xi|\leq \frac 35 \V[X_q]^{\frac 14}$ is at most
$$\int_{-\frac 35 \V[X_q]^{\frac 14}}^{\frac 35 \V[X_q]^{\frac 14}} \frac{ e^{-\frac{\xi^2}2} \bigg(e^{O\left( \frac{\xi^4}{\rho(q)\log q'}\right)} -1\bigg)}{\xi}d\xi \ll \frac 1{\rho(q)\log q'} \int_{\mathbb R} \xi^3e^{-\frac {\xi^2}2} d\xi \ll  \frac 1{\rho(q)\log q'}.$$
We now bound the remaining part of the integral using an argument analogous to Proposition 2.14 of \cite{FiMa}. Fix $0\leq \lambda \leq \frac 56$. By the properties of the Bessel function $J_0(x)$, we have that if $|\xi|>\lambda$, then whatever $\gamma_{\chi}\in \mathbb R$ is,
$$ \left|J_0 \Bigg( \frac{2 \xi}{\sqrt{\frac 14+\gamma_{\chi}^2}} \Bigg) \right| \leq  J_0 \Bigg( \frac{\lambda}{\sqrt{\frac 14+\gamma_{\chi}^2}} \Bigg).$$
By \eqref{equation characteristic function of X_q}, this shows that in the range $|\xi|>\frac 5{12}\V[X_q]^{-\frac 14}$ we have $|\hat X_q(\xi)|\leq |\hat X_q(\frac 5{12}\V[X_q]^{-\frac 14})|$ (since $\V[X_q]\geq 1$), and so
\begin{align*}\int_{\frac 35 \V[X_q]^{\frac 14}<|\xi|\leq \V[X_q]} \frac{\hat Y_q(\xi)-e^{-\frac{\xi^2}2}}{\xi}d\xi &\ll \hat Y_q\left(\frac 5{12} \V[X_q]^{\frac 14}\right) \log\V[X_q] + \int_{|\xi|>\frac 35 \V[X_q]^{\frac 14}} \frac{e^{-\frac{\xi^2}2}}{\xi}d\xi\\
& \ll e^{-\frac{25}{577}\V[X_q]^{\frac 12}}+e^{-\frac{9}{51}\V[X_q]^{\frac 12}},
\end{align*}
by \eqref{equation bound on Y_q}. Applying Lemma \ref{lemma variance of X_q}, we conclude that the right hand side of \eqref{equation Berry-Esseen} is at most a constant times $(\rho(q)\log q')^{-1}$.

\end{proof}

\begin{proof}[Proof of Theorem \ref{theorem extreme Dirichlet races}, second part]
Fix $\eta \in [\frac 12,1]$. We wish to find a sequence of moduli $\{q_n\}$ such that $\delta(q_n,NR,R)\rightarrow \eta$. The case $\eta=1$ was already covered in part (1), and the case $\eta = \frac 12$ follows from taking prime values of $q$, by the central limit theorem of Rubinstein and Sarnak \cite{RubSar}. Therefore we can assume that $\frac 12<\eta<1$. 

Let $\kappa>0$ be the unique real solution of the equation 
$$ \frac 1{\sqrt{2\pi}}\int_{-\kappa}^{\infty} e^{-\frac{t^2}2}dt=\eta.  $$
Let moreover $\{q_n\}$ be the sequence of squarefree odd integers coming from Lemma \ref{lemma prescribed prime factors} for which 
$$ 2^{\omega(q_n)+1} = \log q_n'(\kappa^2+o(1)). $$
By \eqref{equation computation of B(q)}, this gives that as $n\rightarrow \infty,$ 
$$ B(q_n):=\frac{\mathbb E[X_{q_n}]}{\sqrt{\V[X_{q_n}]}}  \longrightarrow\kappa. $$

Define 
$$ Y_{q_n}:= \frac{X_{q_n}-\E[X_{q_n}]}{\sqrt{\V[X_{q_n}]}} =\frac{X_{q_n}}{\sqrt{\V[X_{q_n}]}}-B(q_n).$$

We will use the central limit theorem of Lemma \ref{lemma clt}, as well as the Berry-Essen inequality \eqref{equation Berry-Esseen}. Denoting by $F_{q_n}$ the distribution function of $Y_{q_n}$ and by $F$ that of the Gaussian distribution, we have that

\begin{align*}
|\delta(q_n,NR,R)-\eta| & = |\P[X_{q_n}>0]-\eta| = |\P[X_{q_n}\leq 0]-(1-\eta)| \\
&=|F_{q_n}(-B(q_n)) - F(-\kappa) | \\ &\leq |F_{q_n}(-B(q_n)) - F(-B(q_n)) | + |F(-B(q_n))-F(-\kappa) |\\
& \ll \frac 1{\rho(q_n)\log q_n'} + |\kappa-B(q_n)|,
\end{align*}
by Lemma \ref{lemma berry esseen} and by the fact that the probability density function of the Gaussian is bounded on $\mathbb R$. Looking at the proof of Lemma \ref{lemma prescribed prime factors} we see that $\rho(q_n) \rightarrow \infty$, hence this last quantity tends to zero as $n\rightarrow \infty$, concluding the proof.
\end{proof}

\label{section Dirichlet}

\section{A more precise estimation of the bias using the theory of large deviations}
\label{deviations}

To give a more precise estimate for the bias we are interested in under LI, we use the theory of large deviations of independent random variables. The fundamental estimate of this section is Theorem 2 of \cite{MoOd}.

\begin{theorem}[Montgomery and Odlyzko]
\label{theorem mood}
For $n=1,2,...$ let $Y_n$ be independent real valued random variables such that $\E[Y_n]=0$ and $|Y_n|\leq 1$. Suppose that there is a constant $c>0$ such that $\E[Y_n^2]\geq c$ for all $n$. Put $Y=\sum r_n Y_n$ where $\sum r_n^2 <\infty$.

If $\sum_{|r_n|\geq \alpha} |r_n| \leq V/2$ then
$$ \P[Y\geq V] \leq \exp\bigg( -\frac 1{16} V^2 \bigg( \sum_{|r_n|<\alpha} r_n^2 \bigg)^{-1}\bigg). $$

If $\sum_{|r_n|\geq \alpha} |r_n| \geq 2V$ then
$$ \P[Y\geq V] \geq a_1 \exp\bigg( -a_2 V^2 \bigg( \sum_{|r_n|<\alpha} r_n^2 \bigg)^{-1}\bigg). $$
Here $a_1>0$ and $a_2>0$ depend only on $c$.
\end{theorem}

To make use of these bounds we need to give estimates on sums over zeros.
\begin{lemma}
For $T\geq 1$ we have 
$$ \sum_{\substack{|\gamma_{\chi}| < T }} \frac 1{\sqrt{\frac 14+\gamma_{\chi}^2}} = \frac 1{\pi}\log (q^*\sqrt T) \log T + O(\log (q^*T)).$$
\label{lemma asymptotic for sum of zeros}
\end{lemma}

\begin{proof}
We start from the von Mangoldt formula:
$$ N(T,\chi) = \frac T{\pi} \log \frac{q^*T}{2\pi e} +O(\log q^*T).$$
With a summation by parts we get
\begin{align*}
 \sum_{\substack{|\gamma_{\chi}| < T }} \frac 1{\sqrt{\frac 14+\gamma_{\chi}^2}} &= O(\log q^*)+\int_1^{T} \frac{dN(t,\chi)}{\sqrt{\frac 14+t^2}} \\
&=  \frac{N(T,\chi)}{\sqrt{\frac 14+T^2}}+\int_1^T \frac{tN(t,\chi)}{\left(\frac 14+t^2\right)^{\frac 32}}dt + O(\log q^*) \\
&= \int_1^T \frac{\frac {t^2}{\pi} \log \frac{q^*t}{2\pi e}  }{\left(\frac 14+t^2\right)^{\frac 32}} dt + O(\log (q^*T))\\
&= \frac 1{\pi}\log (q^*\sqrt T) \log T +O(\log q^*T).
\end{align*}
\end{proof}

\begin{lemma}
Let $\F(q)$ be a subset of the invertible residues $\bmod q$ such that $\chi \in \F(q) \Rightarrow \overline{\chi} \in \F(q)$. Define the random variable 
$$ Y :=\sum_{\chi \in \F(q)} \sum_{\gamma_{\chi}>0} \frac{2\Re(Z_{\gamma_{\chi}})}{\sqrt{\frac 14+\gamma_{\chi}^2}},$$
where the $Z_{\gamma_{\chi}}$ are i.i.d. uniformly distributed on the unit circle. Then, we have for $q$ large enough that
$$ a_1 \exp\left(-a_2 \frac{|\F(q)|}{L(q)}\right) \leq  \P[Y\geq|\F(q)|] \leq  \exp\left(-a_3 \frac{|\F(q)|}{L(q)}\right),    $$
where the $a_i$ are absolute constants and
$$ L(q):= \frac{\sum_{\chi \in \F(q)} \log q^*}{|\F(q)|}\geq \frac{\log 2}{2}.$$
\label{lemma Montgomery bounds}
\end{lemma}
\begin{proof}
It is a direct application of Theorem \ref{theorem mood}. Taking the sequence $\{r_i\}$ to be the $\frac 2{\sqrt{\frac 14+\gamma_{\chi}^2}}$ ordered by size, we have for $0<\alpha \leq 4$ that
$$ \sum_{|r_n|\geq \alpha}|r_n| = \sum_{\chi \in \F(q)}\sum_{\substack{ 0<\gamma_{\chi} \leq \sqrt{\frac {4}{\alpha^2}-\frac 14}}} \frac 2{\sqrt{\frac 14+\gamma_{\chi}^2}}, \hspace{1cm}\sum_{|r_n|>\alpha}|r_n|^2= \sum_{\chi \in \F(q)}\sum_{\substack{  \gamma_{\chi} > \sqrt{\frac 4{\alpha^2}-\frac 14}}} \frac 4{\frac 14+\gamma_{\chi}^2}. $$
For the upper bound we take $\alpha=4$: then we trivially have $\sum_{|r_n|\geq \alpha}|r_n|\leq |\F(q)|/2$, so 
$$ \P[Y\geq |\F(q)| ] \leq \exp \left(-\frac 1{16} |\F(q)|^2 \left(c_1 \sum_{\chi \in \F(q)} \log q^*\right)^{-1}  \right) $$
for some absolute constant $c_1$. For the lower bound we take $\alpha=2/\sqrt{\frac 14+T_0^2}$, where $T_0>1$ is a fixed large real number (independent of $q$ and $\mathcal F(q)$) such that 
$$ \sum_{\chi \in \F(q)} \sum_{\substack{|\gamma_{\chi}| \leq T_0 }} \frac 1{\sqrt{\frac 14+\gamma_{\chi}^2}} \geq \frac{4}{\log 2} L(q)|\F(q)| \geq 2|\F(q)|, $$
whose existence is granted by Lemma \ref{lemma asymptotic for sum of zeros} (we grouped together conjugate characters). Then Theorem \ref{theorem mood} gives the bound
\begin{align*}
\P[Y\geq |\F(q)| ] &\geq c_2 \exp \left(-c_3 |\F(q)|^2 \left(\sum_{\chi\in\F(q)} \sum_{\gamma_{\chi}>T_0}\frac 4{\frac 14+\gamma^2}\right)^{-1}  \right)
\\ &\geq c_2 \exp \left(-c_3 |\F(q)|^2 \left(c_4 \sum_{\chi \in \F(q)} \log q^*\right)^{-1}  \right)  
\end{align*}
for $q$ large enough and some absolute constants $c_2,c_3$ and $c_4$, since if we choose $T_1>T_0$ independent of $\chi$ and large enough such that $N(2T_1,\chi)-N(T_1,\chi) \gg \log q^*$ (this is possible by the von-Mangoldt formula), then we have 
\begin{align*}\sum_{\chi\in\F(q)} \sum_{\gamma_{\chi}>T_0}\frac 4{\frac 14+\gamma^2} \geq \sum_{\chi\in\F(q)} \sum_{T_1 <\gamma_{\chi}<2T_1}\frac 4{\frac 14+\gamma^2} &\geq \sum_{\chi\in\F(q)} \frac 4{\frac 14+(2T_1)^2} (N(2T_1,\chi)-N(T_1,\chi)) \\
& \gg   \sum_{\chi\in\F(q)} \log q^*.
\end{align*}

\end{proof}

\begin{proof}[Proof of Theorem \ref{theorem more precise large deviation}]
Let $X_q$ be the random variable in \eqref{equation residues versus non residues random variable} and define the symmetric random variable $$Y_{q}:=X_{q}-\E[X_{q}].$$ By Lemma \ref{lemma link with random variable X_q}
\begin{align*}  \delta(q;NR,R) &= \P[X_{q}>0] \\&= \P[Y_{q}>-\E[X_{q}]]  \\
&= \P[Y_{q}<\E[X_{q}]] = 1-\P[Y_{q}\geq \E[X_{q}]].
\end{align*}
 The proof follows by taking $\F(q):=\{ \chi \bmod q: \chi^2=\chi_0, \chi\neq \chi_0\}$ in Lemma \ref{lemma Montgomery bounds} and by estimating $L(q)$ as in the proof of Lemma \ref{lemma variance of X_q}.

\end{proof}

\section{A more general analysis}

\label{section general analysis}

In this section we do a more general analysis by studying arbitrary linear combinations of prime counting functions.

Throughout the section, $\a=(a_1,...,a_k)$ will be a vector of invertible reduced residues $\bmod q$ and $\al=(\alpha_1,...,\alpha_k)$ will be a non-zero vector of real numbers such that $\sum_{i=1}^k\alpha_i=0$. Recall that
$$ \epsilon_i = \begin{cases}
1 & \text{ if }  a_i \equiv \square \bmod q\\
0 & \text{ if } a_i \not\equiv \square \bmod q,
\end{cases}
$$
and we assume without loss of generality that 
$$ \sum_{i=1}^k \epsilon_i \alpha_i <0. $$

To prove theorems \ref{theorem constant coefficients}, \ref{theorem general way of being biased} and  \ref{theorem limitations generales}, we need a few lemmas.
\begin{lemma}
\label{lemma link with random variables general}
Assume GRH and LI. Then the quantity
$$ E(y;q,\a;\al):=\phi(q)\frac{\alpha_1\pi(e^y;q,a_1)+...+\alpha_k\pi(e^y;q,a_k)}{e^{y/2}/y} $$
has the same distribution as that of the random variable
\begin{equation}
X_{q;\a,\al}:=-\rho(q)\sum_{i=1}^k \epsilon_i \alpha_i+\sum_{\chi\neq \chi_0} |\alpha_1\chi(a_1)+...+\alpha_k\chi(a_k)| \sum_{\gamma_{\chi}>0} \frac{2\Re(Z_{\gamma_{\chi}})}{\sqrt{\frac 14+\gamma_{\chi}^2}},
\label{equation random variable general}
\end{equation}    
where the $Z_{\gamma_{\chi}}$ are independent random variables following a uniform distribution on the unit circle in $\mathbb C$.
\end{lemma}

\begin{remark}
\label{remark back to particular case}
If we take $a_1,...,a_{\phi(q)(1-\rho(q)^{-1})}$ to be the set of all quadratic non-residues $\bmod q$ with $\alpha_1=...=\alpha_{\phi(q)(1-\rho(q)^{-1})}=\frac 1{\phi(q)}$, and we take $a_{\phi(q)(1-\rho(q)^{-1})+1},...,a_{\phi(q)}$ to be the set of all quadratic residues $\bmod q$ with $\alpha_{\phi(q)(1-\rho(q)^{-1})+1}=...=\alpha_{\phi(q)}=\frac{1-\rho(q)}{\phi(q)}$, then we recover the formula \eqref{equation residues versus non residues random variable}.

\end{remark}
\begin{proof}

In the same way as in the proof of Lemma \ref{lemma link with random variable X_q}, we get by the explicit formula and by applying GRH that

\begin{align*}
F(y;q,\a,\al):&= \phi(q)\frac{\alpha_1\psi(e^y;q,a_1)+...+\alpha_k\psi(e^y;q,a_k)}{e^{y/2}} \\&=-\sum_{\chi\neq \chi_0} (\alpha_1\overline{\chi}(a_1)+...+\alpha_k\overline{\chi}(a_k)) \sum_{\gamma_{\chi}} \frac{e^{i\gamma_{\chi}y}}{\rho_{\chi}}+o_q(1),
\end{align*}
(the main terms cancel since $\sum_{i=1}^k\alpha_i=0$). By the work of Rubinstein and Sarnak \cite{RubSar}, $F(y;q,\a,\al)$ has the same distribution as $X_{q;\a,\al}-\E[X_{q;\a,\al}]$, since LI implies that there are no real zeros. The second step is to use summation by parts and to remove squares and other prime powers; this gives that
$$ E(y;q,\a,\al)+\rho(q)\sum_{i=1}^k \epsilon_i \alpha_i+o(1)=F(y;q,\a,\al),$$
completing the proof.
\end{proof}

Before we give a bound on the variance of this distribution, we prove a lemma about conductors.

\begin{lemma}
\label{lemma small conductors}
Let $1\leq L\leq \phi(q)$. Then,
$$ \# \{\chi \bmod q : q^* \leq L \} \leq \min\{ L\tau(q),L^2\}.$$
\end{lemma}
\begin{proof}
Denoting by $\phi^*(d)$ the number of primitive characters $\bmod q$, we have
$$ \sum_{\substack{d\mid q \\ d\leq L}} \phi^*(d) \leq \min \left\{   \sum_{\substack{d\leq L}} d, L \sum_{\substack{d\mid q }} 1 \right\}.  $$
\end{proof}

\begin{lemma}
\label{lemma bounds on variance general}
Assume LI. Let $V(q;\a,\al):=\V[X_{q;\a,\al}]$, where $X_{q;\a,\al}$ is the random variable defined in \eqref{equation random variable general}. Then,
\begin{equation}
 \phi(q) \norm{\a}^2_2 \log \left(\frac {3\phi(q)}k\right) \ll  V(q;\a,\al) \ll \phi(q) \norm{\a}^2_2 \log q,
 \label{equation bounds on the variance general}
\end{equation}
where 
$$ \norm{\a}^2_2 := \sum_{i=1}^k \alpha_i^2.$$
\end{lemma}
\begin{remark}
The upper bound in \eqref{equation bounds on the variance general} is attained when $q$ is prime by Lemma \ref{lemma exact expression for the variance general}. As for the lower bound, if we take moduli $q$ with a bounded number of distinct prime factors and consider the race between residues and non-residues with the weights of Remark \ref{remark back to particular case}, we obtain by Lemma \ref{lemma variance of X_q} that $V(q;\a,\al) =O(1)$, and this is of the same order of magnitude as the lower bound in \eqref{equation bounds on the variance general}.
\end{remark}
\begin{proof}
Since the $Z_{\gamma_{\chi}}$ in \eqref{equation random variable general} are independent and have variance equal to $\frac 12$, we have that 
$$\V[X_{q;\a,\al}] = \sum_{\chi\neq \chi_0} |\alpha_1\chi(a_1)+...+\alpha_k\chi(a_k)|^2 \sum_{\gamma_{\chi}} \frac{1}{\frac 14+\gamma_{\chi}^2}. $$
(LI implies that there are no real zeros.)
By \eqref{equation sum over zeros B(chi)}, there exists $q_0$ such that whenever $q^*\geq q_0$ we have
\begin{equation}
\sum_{\gamma_{\chi}} \frac{1}{\frac 14+\gamma_{\chi}^2} \asymp \log q^*,
\label{equation pre variance}
\end{equation}
and the same estimate clearly holds for $q^*<q_0$, since the left hand side of \eqref{equation pre variance} is positive. We conclude that

\begin{equation}
V(q;\a,\al) \asymp \sum_{\chi\neq \chi_0} |\alpha_1\chi(a_1)+...+\alpha_k\chi(a_k)|^2 \log q^*.
\label{equation first approximation to variance general}
\end{equation}  
Now, $\alpha_1\chi(a_1)+...+\alpha_k\chi(a_k)=0$, so
\begin{align*}
 \sum_{\chi\neq \chi_0} |\alpha_1\chi(a_1)+...+\alpha_k\chi(a_k)|^2 &=\sum_{\chi\bmod q} |\alpha_1\chi(a_1)+...+\alpha_k\chi(a_k)|^2 \\
 &= \sum_{1\leq i,j \leq k} \alpha_i \alpha_j \sum_{\chi\bmod q} \chi(a_ia_j^{-1}) \\
 &=\phi(q) \sum_{i=1}^k \alpha_i^2. 
\end{align*}

Using this and \eqref{equation first approximation to variance general}, the upper bound follows from the fact that $\log q^*\leq \log q$. This also gives the lower bound $V(q;\a,\al)\geq \log 3 \phi(q) \norm{\al}_2^2$, which proves the claim for bounded values of $\phi(q)/k$. Hence we assume from now on that $\phi(q)/k\geq 576$. We fix a parameter $1<L<\phi(q)$ and discard the characters of conductor at most $L$:
\begin{align*}
 V(q;\a,\al) &\geq \log L \sum_{\substack{\chi \bmod q: \\ q^* > L}} |\alpha_1\chi(a_1)+...+\alpha_k\chi(a_k)|^2 \\
 & = \log L \sum_{1\leq i,j\leq k} \alpha_i \alpha_j  \sum_{\substack{\chi \bmod q: \\ q^* > L}} \chi(a_ia_j^{-1}) \\
 &= \log L \left[ \sum_{i=1}^k \alpha_i^2 \sum_{\substack{\chi\bmod q \\ q^*> L}} 1 + \sum_{1\leq i\neq j \leq k} \alpha_i\alpha_j\sum_{\substack{\chi \bmod q: \\ q^* > L}} \chi(a_ia_j^{-1})  \right],
 \end{align*}
which by Lemma \ref{lemma small conductors} and the orthogonality relations is
\begin{align*} &\geq \log L \left[ \sum_{i=1}^k \alpha_i^2 (\phi(q)-\min\{ L\tau(q),L^2\})  - \sum_{1\leq i\neq j \leq k} |\alpha_i\alpha_j| \min\{ L\tau(q),L^2\}  \right] \\
& \geq \log L \norm{\al}_2^2 \left[ \phi(q)-(k+1)\min\{ L\tau(q),L^2\}  \right]
\end{align*}
by the Cauchy-Schwartz inequality. Taking $L:=(3\phi(q)/k)^{\frac 13}$ gives the result, since then $\phi(q)/k\geq 576$ implies that $(k+1)L^2 \leq \phi(q)/2$.

\end{proof}
\begin{remark}
\label{remark lemma variance general}
In the last proof, we did not lose a lot by discarding the characters of conductor at most $(3\phi(q)/k)^{\frac 13}$, since by Lemma \ref{lemma small conductors} and the Cauchy-Schwartz inequality, their contribution is 
$$ \ll \phi(q) \norm{\al}_2^2 \log \left(\frac{3\phi(q)}k\right).$$

\end{remark}

\begin{proof}[Proof of Theorem \ref{theorem general way of being biased}]
We have by Lemma \ref{lemma bounds on variance general} that there exists an absolute constant $c>0$ such that
$$ B(q;\a,\al):= \frac{\E[X_{q;\a,\a}]}{\sqrt{\V[X_{q;\a,\a}]}} \geq \frac{\rho(q) \left|\sum_{i=1}^k \epsilon_i\alpha_i\right|}{ \sqrt{c\phi(q) \log q \sum_{i=1}^k \alpha_i^2 }},$$
a quantity which is greater or equal to $(c\epsilon)^{-\frac 12}$ by the condition of the theorem. We conclude that $1-\delta(q;\a,\al)\leq c \epsilon$ by using Chebyshev's bound in the same way as in the proof of Theorem \ref{theorem extreme Dirichlet races}.

\end{proof}

\begin{proof}[Proof of Theorem \ref{theorem constant coefficients}]
It is a particular case of Theorem \ref{theorem general way of being biased}.
\end{proof}

We now prove our negative results. To do so, we need to provide a central limit theorem, analogous to Lemma \ref{lemma clt}.
\begin{lemma}
Let 
$$Y_{q;\a,\al}:=\frac {X_{q;\a,\al}-\E[X_{q;\a,\al}]}{\sqrt{\V[X_{q;\a,\al}]}}.$$
The characteristic function of $Y_{q;\a,\al}$ satisfies
$$ \hat Y_{q;\a,\al}(\xi)= -\frac{\xi^2}2+O\left( \frac{\xi^4}{\log(3\phi(q)/k)} \min\left\{ 1,\frac{k^2\log q}{\phi(q) \log (3\phi(q)/k) } \right\}\right) $$
in the range $|\xi|\leq \frac 3{5 \norm{\al}_1}$, where $\norm{\al}_1:=\sum_{i=1}^k |\alpha_i|$.
\label{lemma clt general}
\end{lemma}
\begin{proof}
As in Lemma \ref{lemma clt}, we compute
$$ \log \hat X_{q;\a,\al}(\xi) =i \E[X_{q;\a,\al}] \xi+ \sum_{\chi\neq \chi_0} \sum_{\gamma_{\chi}>0} \log \left(J_0 \left( \frac{2|\alpha_1\chi(a_1)+...+\alpha_k\chi(a_k)| \xi}{\sqrt{\frac 14+\gamma^2}} \right) \right). $$
We now use the Taylor expansion \eqref{equation Taylor expansion of Bessel}, which is valid as soon as $|\xi|\leq \frac 3{5 \norm{\al}_1}$, since under this condition we have
$$ \frac{2|\alpha_1\chi(a_1)+...+\alpha_k\chi(a_k)| |\xi|}{\sqrt{\frac 14+\gamma^2}} \leq \frac{2\norm{\al}_1}{1/2} \frac 3{5\norm{\al}_1}=\frac{12}5.$$
Using \eqref{equation pre variance} and the analogous estimate
$\sum_{\gamma_{\chi}} (\frac 14+\gamma_{\chi}^2)^{-2} \asymp \log q^*,$
we get
\begin{equation} \log \hat Y_{q;\a,\al}(\xi) = -\frac{\xi^2}2 + O\left( \xi^4 \frac{\sum_{\chi\neq \chi_0}|\alpha_1\chi(a_1)+...+\alpha_k\chi(a_k)|^4\log q^* }{\left(\sum_{\chi\neq \chi_0}|\alpha_1\chi(a_1)+...+\alpha_k\chi(a_k)|^2\log q^*\right)^2}\right).
\label{equation characteristic function general}
\end{equation}

If $\phi(q)/k$ is bounded, then the statement trivially follows from the bound $\sum_i a_i^4 \leq (\sum_i a_i^2)^2$. Therefore we assume from now on that $\phi(q)/k\geq 576$.

We now use two different approaches to bound the error term. The first idea is to "factor out $\sqrt{\log q^*}$" before applying the trivial inequality $\sum_i a_i^4 \leq (\sum_i a_i^2)^2$. We have seen in Remark \ref{remark lemma variance general} that the main contribution to the variance is that of the characters with $q^*\geq L:=(3\phi(q)/k)^{\frac 13}$. We use the same idea here. Setting $\Theta_{\chi}:=|\alpha_1\chi(a_1)+...+\alpha_k\chi(a_k)|^2$, we have
\begin{align} \begin{split}\sum_{\chi\neq \chi_0}\Theta_{\chi}\log q^* &\geq \sum_{\substack{\chi\neq \chi_0 \\ q^* > L}}\Theta_{\chi}\log q^* \geq \sqrt{\log L} \sum_{\substack{\chi\neq \chi_0 \\ q^* > L}}\Theta_{\chi}\sqrt{\log q^*} \\
& \geq \sqrt{\log L} \left(\sum_{\substack{\chi\neq \chi_0 }}\Theta_{\chi} \sqrt{\log q^*} - k L^2\sqrt{\log L} \norm{\al}_2^2 \right)
\end{split}
\label{equation square root log L}
\end{align}
by Lemma \ref{lemma small conductors} and the Cauchy-Schwartz inequality. Now,
by our choice of $L$ and by the fact that $\phi(q)/k \geq 576$ we have
\begin{align*}
  kL^2 \sqrt{\log L} \norm{\al}_2^2 &\leq \frac 12 \sqrt{\log L} \bigg[\phi(q)\norm{\al}_2^2 -kL^2 \norm{\al}_2^2 \bigg] \\&\leq \frac 12 \sum_{\substack{\chi \neq \chi_0 \\ q^*\geq L}} \Theta_{\chi} \sqrt{\log q^*} \leq 
\frac 12 \sum_{\chi \neq \chi_0} \Theta_{\chi} \sqrt{\log q^*},
\end{align*}
hence \eqref{equation square root log L} gives that
$$\sum_{\chi\neq \chi_0}\Theta_{\chi}\log q^* \gg \sqrt{\log L}\sum_{\substack{\chi\neq \chi_0 }}\Theta_{\chi} \sqrt{\log q^*}. $$
Plugging this into \eqref{equation characteristic function general} and using the trivial bound $\sum_{\substack{\chi\neq \chi_0 }}\Theta_{\chi}^2 \log q^* \leq \left(\sum_{\substack{\chi\neq \chi_0 }}\Theta_{\chi} \sqrt{\log q^*} \right)^2$, we get that the error term is $\ll \xi^4 / \log L$.

For the second upper bound we use Lemma \ref{lemma bounds on variance general} and the Cauchy-Schwartz inequality:

\begin{align*} \frac{\sum_{\chi\neq \chi_0}|\alpha_1\chi(a_1)+...+\alpha_k\chi(a_k)|^4\log q^* }{\left(\sum_{\chi\neq \chi_0}|\alpha_1\chi(a_1)+...+\alpha_k\chi(a_k)|^2\log q^*\right)^2} &\ll \frac{\log q}{\log (3\phi(q)/k)^2} \frac{\sum_{\chi\neq \chi_0}|\alpha_1\chi(a_1)+...+\alpha_k\chi(a_k)|^4 }{(\phi(q) \norm{\al}_2^2)^2} \\
& = \frac{\log q}{\log (3\phi(q)/k)^2} \frac{\sum_{\substack{i,j,i',j' \\ a_{i}a_{i'} \equiv a_{j}a_{j'} \bmod q}} \alpha_i\alpha_{i'}\alpha_j\alpha_{j'} }{\phi(q) \norm{\al}_2^4 }\\
& \leq \frac{\log q}{\log (3\phi(q)/k)^2} \frac{\left(\sqrt{\sum_{i=1}^k \alpha_i}\sqrt{ \sum_{j=1}^k 1}\right)^4}{\phi(q) \norm{\al}_2^4 },  
\end{align*} 
which gives the claimed bound.
\end{proof}

\begin{proof}[Proof of Theorem \ref{theorem limitations generales}]
Let $K\geq 1$ and define $c>0$ to be the constant implied in the lower bound in Lemma \ref{lemma bounds on variance general}. Assume that $k\leq e^{-e^{4K}} \phi(q)$ and that \eqref{equation hypothesis thm not biased} holds with $K_2=K$. Define the vector $\be:=\frac{e^{-K}}{\norm{\al}_1} \al$, so that $\norm{\be}_1 = e^{-K}$, which will allow us to apply Lemma \ref{lemma clt general}. Clearly,
$$ \delta(q;\a,\al)=\delta(q;\a,\be),$$
since multiplying $\al$ by a positive constant does not affect the inequality $\alpha_1\pi(n;q,a_1)+...+\alpha_k\pi(n;q,a_k) >0$. 

We have by Lemma \ref{lemma bounds on variance general} and by the definition of $c$ that
\begin{align*} B(q;\a,\be):= \frac{\E[X_{q;\a,\be}]}{\sqrt{\V[X_{q;\a,\be}]}} &\leq \frac{\rho(q)\left|\sum_{i=1}^k \epsilon_i\beta_i \right|}{ \sqrt{c\phi(q) \log \left(3\phi(q)/k\right) \sum_{i=1}^k \beta_i^2 }} \\& = c^{-\frac 12} \frac{\rho(q)\left|\sum_{i=1}^k \epsilon_i\alpha_i\right|}{ \sqrt{\phi(q) \log \left(3\phi(q)/k\right) \sum_{i=1}^k \alpha_i^2 }},
\end{align*}
a quantity which is at most $\sqrt {K}$ by \eqref{equation hypothesis thm not biased}. Defining
$$Y_{q;\a,\be}:=\frac {X_{q;\a,\be}-\E[X_{q;\a,\be}]}{\sqrt{\V[X_{q;\a,\be}]}},$$
we have by Lemma \ref{lemma clt general} and by our condition on $k$ that in the range $|\xi|\leq \frac 3{5}e^{K}$,
$$ \log \hat Y_{q;\a,\be}(\xi)= -\frac{\xi^2}2+O\left( \frac{\xi^4}{e^{4K}} \right). $$
Combining this with the Berry-Esseen inequality \eqref{equation Berry-Esseen} and taking $W$ to be a standard Gaussian random variable with mean $0$ and variance $1$ we get
\begin{align} 
\begin{split}\P[Y_{q;\a,\be} > -B(q;\a,\be)] &- \P[W>-B(q;\a,\be)]  \\ &\ll \int_{-\frac 35 e^{K}}^{\frac 35 e^{K}} \frac{\hat Y_{q;\a,\be}(\xi)-e^{-\frac{\xi^2}{2}}}{\xi} d\xi + \frac 53e^{-K} \\
&\ll  \int_{-\frac 35 e^{K}}^{\frac 35e^{K}} \frac{\xi^3 e^{-\frac{\xi^2}{2}} }{ e^{4K}} d\xi + e^{-K}\\
&\ll e^{-K}.
\end{split}
\label{equation bound comparing Y to Gaussian}
\end{align}
However, since $B(q;\a,\be) \leq \sqrt K$, we have that 
$$\P[W\leq -B(q;\a,\be)] \geq c_1 \frac{e^{-\frac K2}}K $$
for some absolute constant $c_1$. Therefore, applying \eqref{equation bound comparing Y to Gaussian} gives
\begin{align*}
 \delta(q;\a,\be)&= \P[Y_{q;\a,\be} > -B(q;\a,\be)] \\&= \P[W>-B(q;\a,\be)] +O(e^{-K})
 \\&\leq  1-c_1 e^{-\frac K2}/K +c_2e^{-K},
 \end{align*}
 a quantity which is less than the right hand side of \eqref{equation conclusion of theorem not biased} for $K$ large enough. The proof is finished since $\delta(q;\a,\al)=\delta(q;\a,\be)$.
\end{proof}

To end this section we give an exact expression for the variance $V(q;\a,\al)$. While we have not explicitly made use of this expression, we include it for its intrinsic interest, and for its ability to give a precise evaluation of the variance $V(q;\a,\al)$ for values of $q$ having prescribed prime factors.  
\begin{lemma}
\label{lemma exact expression for the variance general}
We have that
\begin{equation}
V(q;\a,\al) = \phi(q) \norm{\al}_2  \left( \log q- \sum_{p\mid q} \frac{\log p}{p-1} \right) - \phi(q)\sum_{i\neq j} \alpha_i\alpha_j \frac{\Lambda\left( \frac q{(q,a_ia_j^{-1}-1)}\right)}{\phi\left( \frac q{(q,a_ia_j^{-1}-1)}\right)}.
\label{equation exact evaluation variance}
\end{equation} 
\end{lemma}
\begin{proof}
Using Proposition 3.3 of \cite{FiMa}, we get
\begin{align*}
V(q;\a,\al) &= \sum_{\chi \bmod q} |\alpha_1\chi(a_1)+...+\alpha_k\chi(a_k)|^2 \log q^* \\
&= \sum_{1\leq i,j \leq k} \alpha_i\alpha_j \sum_{\chi \bmod q} \chi(a_ia_j^{-1}) \log q^* \\
&=  \phi(q)\left( \log q- \sum_{p\mid q} \frac{\log p}{p-1} \right) \sum_{i=1}^k \alpha_i^2 -\phi(q) \sum_{i\neq j} \alpha_i\alpha_j \frac{\Lambda\left( \frac q{(q,a_ia_j^{-1}-1)}\right)}{\phi\left( \frac q{(q,a_ia_j^{-1}-1)}\right)}. 
\end{align*}

\end{proof}
It might seem like the second term of \eqref{equation exact evaluation variance} is an error term, however this is not necessarily true for large values of $k$ (see Lemma \ref{lemma variance of X_q}). Nevertheless, we expect many cancellations to occur since
$$ \sum_{i\neq j} \alpha_i\alpha_j = \left(\sum_{i=1}^k \alpha_i \right)^2-\sum_{i=1}^k \alpha_i^2 =-\sum_{i=1}^k \alpha_i^2.  $$

\end{document}